\documentclass[12pt]{article}
\usepackage{latexsym,amsfonts,amsmath,graphics}

\newtheorem{theorem}{Theorem}
\newtheorem{lemma}{Lemma}
\newtheorem{corollary}{Corollary}

\newtheorem{remark}{Remark}

\newcommand{\wal}{{\rm wal}}

\newcommand{\icomp}{\mathtt{i}}

\newcommand{\rd}{\,\mathrm{d}}
\newcommand{\de}{\mathrm{e}}
\newcommand{\NN}{\mathbb{N}}
\newcommand{\LL}{\mathcal{L}}
\newcommand{\KK}{\mathcal{K}}
\newcommand{\HH}{\mathcal{H}}

\newcommand{\integer}{\mathbb{Z}}
\newcommand{\real}{\mathbb{R}}

\newcommand{\satop}[2]{\stackrel{\scriptstyle{#1}}{\scriptstyle{#2}}}

\allowdisplaybreaks

\newenvironment{proof}{\begin{trivlist}
    \item[\hskip\labelsep{\it Proof.}]}{$\hfill\Box$\end{trivlist}}

\begin{document}

\title{The decay of the Walsh coefficients of smooth functions}

\author{Josef Dick\thanks{School of Mathematics and Statistics, University of New
South Wales, Sydney 2052, Australia. ({\tt
josef.dick@unsw.edu.au})}}

\date{}

\maketitle

\begin{abstract}
We give upper bounds on the Walsh coefficients of functions for
which the derivative of order at least one has bounded variation of
fractional order. Further, we also consider the Walsh coefficients
of functions in periodic and non-periodic reproducing kernel Hilbert
spaces. A lower bound which shows that our results are best possible
is also shown.
\end{abstract}

\bigskip


\bigskip

\section{Introduction}

In this paper we analyze the decay of the Walsh coefficients of
smooth functions. Walsh functions $\wal_k: [0,1) \rightarrow \{1,
\omega_b, \ldots, \omega_b^{b-1}\}$, where $k$ is a non-negative
integer and $\omega_b = \de^{2\pi\icomp/b}$, were first introduced
in \cite{walsh} and further early results were obtained in
\cite{chrest, Fine}. See for example \cite{wade} for an overview. It
is well known that Walsh functions form a complete orthonormal
system of $L_2([0,1))$, see \cite{chrest,niederdrenk}.

In analogy to our aim for Walsh functions, consider Fourier series
for a moment: A classical result says that the $k$th Fourier
coefficient of an $r$ times differentiable function decays with
order $|k|^{-r}$. An analogous result for the Walsh coefficients of
$r$ times differentiable functions has been missing in the
literature and will be provided here.

The Walsh coefficients of functions which satisfy a H\"older
condition were already considered in \cite{Fine}. Here we consider
the decay of the Walsh coefficients of functions which satisfy even
stronger smoothness assumptions, i.e., have at least one smooth
derivative. It has long been known from \cite{Fine} that the only
absolutely continuous functions for which all Walsh coefficients
decay faster than $1/k$ are the constants. Here we refine this
result be showing that for $r$ times differentiable functions, the
Walsh coefficients decay with order $b^{-a_1-\cdots -
a_{\min(v,r)}}$, where $k = \kappa_1 b^{a_1-1} + \cdots + \kappa_v
b^{a_v-1}$ with $0 < \kappa_1, \ldots, \kappa_v < b$ and $a_1 >
\cdots > a_v > 0$. I.e., only the coefficients $k$ which have only
one non-zero digit in their base $b$ expansion decay with order
$1/k$, the others decay faster. We also prove a lower bound which
shows that this result is best possible.

The question of how the Walsh coefficients of smooth functions decay
plays a central role in numerical integration of smooth functions.
In \cite{Dick05,Dick06} this decay was implicitly used to give
explicit constructions of quasi-Monte Carlo rules which achieve the
optimal rate of convergence for the numerical integration of
functions with smoothness $r > 1$.

Throughout the paper  we use the following notation: We assume that
$b \ge 2$ is a natural number, and that $k \in \NN$ (where $\NN$
denotes the set of natural numbers) has base $b$ expansion $k =
\kappa_1 b^{a_1-1} + \cdots + \kappa_v b^{a_v-1}$, where $v \ge 1$,
$0 < \kappa_1,\ldots,  \kappa_v < b$, and $a_1 > \cdots > a_v > 0$.
For $k = 0$ we will assume that $v = 0$.

Let the real number $x \in [0,1)$ have base
$b$ representation $x = \frac{x_1}{b} + \frac{x_2}{b^2} + \cdots$, with $0 \le x_i < b$ and
where infinitely many $x_i$ are different from $b-1$. For $k \in \NN$ we define the $k$th Walsh function by
$$\wal_k(x) = \omega_b^{\kappa_1 x_{a_1} + \cdots + \kappa_v x_{a_v}},$$ where
$\omega_b = \de^{2\pi \icomp/b}$. For $k = 0$ we set $\wal_0(x) = 1$.

For a function $f:[0,1] \rightarrow \real$ we define the $k$th Walsh
coefficient of $f$ by $$\hat{f}(k) = \int_0^1 f(x)
\overline{\wal_k(x)} \rd x$$ and we can form the Walsh series $$f(x)
\sim \sum_{k=0}^\infty \hat{f}(k) \wal_k(x).$$

Among other things, it was shown in \cite{Dick06} that if a function $f:[0,1]\rightarrow \real$ has $r-1$ derivatives for which $f^{(r-1)}$ satisfies a Lipschitz condition, then $|\hat{f}(k)| \le  C_{r} b^{-a_1-\cdots - a_{\min(r,v)}}$ for some constant $C_{r} > 0$ independent of $f$ and $k$. An explicit constant was also given in \cite{Dick06}.

In this paper we improve upon the results in \cite{Dick05} and \cite{Dick06} in several ways. We improve the constant $C_r$ mentioned above and obtain also a constant for $r = \infty$ (in \cite{Dick06} we have $C_r \rightarrow \infty$ for $r \rightarrow \infty$). (In the context of numerical integration this is interesting as we want to know how the
integration error depends on the smoothness.) If the function and all its derivatives are periodic then the result
can be strengthened. This was already implicitly used in \cite{Dick05}, but will be explicitly shown here.

We need the following lemma which was first shown in \cite{Fine} and
appeared in many other papers (see for example \cite{Dick06} for a more general version). The following notation will be used throughout the paper: $k' = k - \kappa_1 b^{a_1-1}$ and hence $0 \le k' < b^{a_1-1}$.

\begin{lemma}\label{lem_Jk}
For $k \in \NN$ let $J_k(x) =
\int_0^x \overline{\wal_k(t)} \rd t$. Then
\begin{eqnarray*}
J_k(x) & = & b^{-a_1} \Bigg((1-\omega_b^{-\kappa_1})^{-1}
\overline{\wal_{k'}(x)} + (1/2 + (\omega_b^{-\kappa_1}-1)^{-1})
\overline{\wal_k(x)} \\ && \qquad + \sum_{c=1}^\infty \sum_{\vartheta =
1}^{b-1} b^{-c} (\omega_b^\vartheta -1)^{-1} \overline{\wal_{\vartheta
b^{a_1+c-1}+k}(x)} \Bigg).
\end{eqnarray*}
For $k = 0$, i.e., $J_0(x) = \int_0^x 1 \rd t = x$, we have
\begin{equation}\label{eq_J0}
J_0(x) = 1/2 + \sum_{c=1}^\infty \sum_{\vartheta=1}^{b-1} b^{-c}
(\omega_b^\vartheta-1)^{-1} \overline{\wal_{\vartheta b^{c-1}}(x)}.
\end{equation}
\end{lemma}

We also need the following elementary lemma.
\begin{lemma}\label{lem_elementary}
For any $0 < \kappa < b$ we have $$|1 - \omega_b^{-\kappa}|^{-1} \le \frac{1}{2 \sin \frac{\pi}{b}} \quad \mbox{and} \quad |1/2 + (\omega_b^{-\kappa}-1)^{-1}| \le \frac{1}{2 \sin\frac{\pi}{b}}.$$
\end{lemma}

We introduce some further notation which will be used throughout the
paper: For $v > 1$ let $k'' = k' - \kappa_2 b^{a_2-1}$, and hence $0
\le k'' < b^{a_2-1}$. For $l \in \NN$ let $l = \lambda_1 b^{d_1-1} +
\cdots + \lambda_w b^{d_w-1}$, where $w \ge 1$, $0 <
\lambda_1,\ldots,  \lambda_w < b$, and $d_1 > \cdots > d_w > 0$.
Further let $l' = l - \lambda_1 b^{d_1-1}$ and hence $0 \le l' <
b^{d_1-1}$. For $w > 1$ let $l'' = l' - \lambda_2 b^{d_2-1}$, and
hence $0 \le l'' < b^{d_2-1}$.

\section{On the Walsh coefficients of polynomials and power series}

In the following we obtain bounds on the Walsh coefficients of
monomials $x^r$.  Let
\begin{equation*}
\chi_{r,v}(a_1,\ldots, a_v;\kappa_1,\ldots,\kappa_v) = \int_0^1 x^r
\overline{\wal_{k}(x)} \rd x.
\end{equation*}
For $k = 0$ we define $\chi_{r,0}$, which is given by $$\chi_{r,0} =
\int_0^1 x^r \rd x = \frac{1}{r+1}.$$

We know from \cite[Lemma~3.7]{Dick06} that the Walsh coefficients of
$x^r$ are $0$ if $v > r$, hence we have $\chi_{r,v} = 0$ for $v > r$.

The Walsh series for $x$ is already known from Lemma~\ref{lem_Jk},
thus (note that we need to take the complex conjugate of (\ref{eq_J0}) to obtain the Walsh series for $x$)
\begin{equation}\label{eq_u11}
\chi_{1,1}(a_1;\kappa_1) = - b^{-a_1} (1- \omega_b^{-\kappa_1})^{-1}.
\end{equation}
It can be checked that $|\chi_{1,1}| \le \frac{1}{2}$. Indeed, we
always have
\begin{equation*}
|\chi_{r,v}(a_1,\ldots,a_v;\kappa_1,\ldots, \kappa_v)| \le \int_0^1 x^r
|\overline{\wal_k(x)}|\rd x = \int_0^1 x^r \rd x = \frac{1}{r+1}
\end{equation*}
for all $r,v \ge 0$.

We obtain a recursive formula for the $\chi_{r,v}$ using integration by
parts, namely
\begin{equation}\label{eq_xrJk}
\int_0^1 x^r \overline{\wal_k(x)} \rd x =  J_k(x) x^r \mid_0^1 - r
\int_0^1 x^{r-1} J_k(x) \rd x = - r \int_0^1 x^{r-1} J_k(x) \rd x.
\end{equation}

Using Lemma~\ref{lem_Jk} and (\ref{eq_xrJk}) we obtain for $1 \le v
\le r$ and $r > 1$ that
\begin{eqnarray}\label{eq_recursion}
\lefteqn{ \chi_{r,v}(a_1,\ldots, a_v; \kappa_1,\ldots, \kappa_v) } \\ &
= & -r b^{-a_1} \bigg((1-\omega_b^{-\kappa_1})^{-1}
\chi_{r-1,v-1}(a_2,\ldots,a_v;\kappa_2,\ldots, \kappa_v) \nonumber \\
&& \qquad\qquad + (1/2+(\omega_b^{-\kappa_1}-1)^{-1}) \chi_{r-1,v}(a_1,
\ldots, a_v;\kappa_1, \ldots, \kappa_v) \nonumber \\ && \qquad\qquad
+ \sum_{c=1}^\infty \sum_{\vartheta=1}^{b-1} b^{-c} (\omega_b^\vartheta-1)^{-1}
\chi_{r-1,v+1}(a_1+c, a_1,\ldots,a_v; \vartheta,
\kappa_1,\ldots,\kappa_v)\bigg). \nonumber
\end{eqnarray}

From (\ref{eq_recursion}) we can obtain
\begin{equation*}
\chi_{r,r}(a_1,\ldots, a_r;\kappa_1,\ldots, \kappa_r) = (-1)^r r!
b^{-a_1-\cdots - a_r} \prod_{w=1}^r (1-\omega_b^{-\kappa_w})^{-1}
\end{equation*}
and, with a bit more effort,
\begin{eqnarray*}
\lefteqn{ \chi_{r,r-1}(a_1,\ldots,a_{r-1};\kappa_1,\ldots,
\kappa_{r-1}) } \\ & = & (-1)^{r} r! b^{-a_1-\cdots - a_{r-1}}
\prod_{w=1}^{r-1} (1-\omega_b^{-\kappa_w})^{-1} \\ && \times (-1/2 +
\sum_{w=1}^{r-1} (1/2+(\omega_b^{-\kappa_w}-1)^{-1}) b^{-a_w}),
\end{eqnarray*}
for all $r \ge 1$.

In principle we can obtain all values of $\chi_{r,v}$ recursively using
(\ref{eq_recursion}). We calculated already $\chi_{r,v}$ for $v = r,
r-1$ and we could continue doing so for $v = r-2, \ldots, 1$. But
the formulae become increasingly complex, so we only prove a bound
on them.

For any $r \ge 0$ and a non-negative integer $k$ we define
\begin{equation*}
\mu_r(k) = \left\{\begin{array}{ll} 0 & \mbox{for } r = 0, k \ge 0,
\\ 0 & \mbox{for } k = 0, r \ge 0, \\ a_1 + \cdots + a_v & \mbox{for
} 1 \le v \le r, \\ a_1 + \cdots + a_r & \mbox{for } v > r.
\end{array} \right.
\end{equation*}

\begin{lemma}\label{lem_bound_Walsh_xr}
For $1 \le r < v$ we have $\chi_{r,v} = 0$ and for any $1 \le v \le r$ we have
\begin{eqnarray*}
\lefteqn{ |\chi_{r,v}(a_1,\ldots, a_v;\kappa_1,\ldots, \kappa_v)| } \\ & \le &  \min_{0 \le u \le v} b^{-\mu_u(k)} \frac{r!}{(r-u+1)!}  \frac{3^{\min(1,u)}}{(2\sin\frac{\pi}{b})^u} \left(1 + \frac{1}{b}+ \frac{1}{b(b+1)} \right)^{\max(0, u-1)}.
\end{eqnarray*}

\end{lemma}

\begin{proof}
The first result was already shown in \cite{Dick06}.

For the second result we use induction on $r$. We have
already shown the result for $r = v =  1$.

Now assume that
\begin{eqnarray*}
\lefteqn{ |\chi_{r-1,v}(a_1,\ldots, a_v;\kappa_1,\ldots, \kappa_v)| } \\ & \le &  \min_{0
\le u \le v}b^{-\mu_u(k)}  \frac{(r-1)!}{(r-u)!}
\frac{3^{\min(1,u)}}{(2\sin\frac{\pi}{b})^u} \left(1 + \frac{1}{b} + \frac{1}{b(b+1)} \right)^{\max(0,u-1)}.
\end{eqnarray*}
We show that the result holds for $r$. We have already shown that
$|\chi_{r,v}| \le \frac{1}{r+1}$, which proves the result for $u = 0$.

By taking the absolute value of (\ref{eq_recursion}) and using the triangular inequality we obtain
\begin{eqnarray}\label{eq_recursion_abv}
\lefteqn{ |\chi_{r,v}(a_1,\ldots, a_v; \kappa_1,\ldots, \kappa_v)| } \\ &
\le & r b^{-a_1} \bigg(|1-\omega_b^{-\kappa_1}|^{-1}
|\chi_{r-1,v-1}(a_2,\ldots,a_v;\kappa_2,\ldots, \kappa_v)| \nonumber \\
&& \qquad + |1/2+(\omega_b^{-\kappa_1}-1)^{-1}| |\chi_{r-1,v}(a_1,
\ldots, a_v;\kappa_1, \ldots, \kappa_v)| \nonumber \\ && \qquad
+ \sum_{c=1}^\infty \sum_{\vartheta=1}^{b-1} b^{-c} |\omega_b^\vartheta-1|^{-1}
|\chi_{r-1,v+1}(a_1+c, a_1,\ldots,a_v; \vartheta,
\kappa_1,\ldots,\kappa_v)|\bigg). \nonumber
\end{eqnarray}

Using Lemma~\ref{lem_elementary}, $|\chi_{r-1,v}| \le \frac{1}{r}$ and $\sum_{c=1}^\infty b^{-c} = \frac{1}{b-1}$, we obtain from (\ref{eq_recursion_abv}) that
\begin{equation*}
|\chi_{r,v}(a_1,\ldots, a_v;\kappa_1,\ldots, \kappa_v)| \le \frac{3 b^{-a_1}}{2 \sin\frac{\pi}{b}},
\end{equation*}
which proves the bound for $u = 1$.

To prove the bound for  $1 < u \le v$ we proceed in the same manner. Using Lemma~\ref{lem_elementary}, and
\begin{eqnarray*}
\lefteqn{|\chi_{r-1,v}(a_1, \ldots, a_v; \kappa_1, \ldots, \kappa_v)| } \\ & \le & b^{-\mu_{u-1}(k)}  \frac{(r-1)!}{(r-u+1)!}
\frac{3^{\min(1,u-1)}}{(2\sin\frac{\pi}{b})^{u-1}} \left(1 + \frac{1}{b} + \frac{1}{b(b+1)} \right)^{\max(0,u-2)},
\end{eqnarray*}
we obtain
\begin{eqnarray*}
\lefteqn{ |\chi_{r,v}(a_1,\ldots, a_v; \kappa_1,\ldots, \kappa_v)| } \\ &
\le & \frac{r b^{-a_1}}{2 \sin\frac{\pi}{b}} \bigg(
|\chi_{r-1,v-1}(a_2,\ldots,a_v;\kappa_2,\ldots, \kappa_v)| \\ && \qquad\qquad + |\chi_{r-1,v}(a_1,
\ldots, a_v;\kappa_1, \ldots, \kappa_v)| \\ && \qquad\qquad
+ \sum_{c=1}^\infty \sum_{\vartheta=1}^{b-1} b^{-c}
|\chi_{r-1,v+1}(a_1+c, a_1,\ldots,a_v; \vartheta,
\kappa_1,\ldots,\kappa_v)|\bigg) \\ & \le & b^{-\mu_u(k)} \frac{r!}{(r-u+1)!}  \frac{3^{\min(1,u)}}{(2\sin\frac{\pi}{b})^u} \left(1 + \frac{1}{b}+ \frac{1}{b(b+1)} \right)^{\max(0, u-2)} \\ && \qquad\qquad \times \left(1 + b^{a_2-a_1} + \frac{b^{a_2-a_1}}{b+1}\right) \\ & \le & b^{-\mu_u(k)} \frac{r!}{(r-u+1)!}  \frac{3^{\min(1,u)}}{(2\sin\frac{\pi}{b})^u} \left(1 + \frac{1}{b}+ \frac{1}{b(b+1)} \right)^{\max(0, u-1)},
\end{eqnarray*}
as $\sum_{c=1}^\infty \sum_{\vartheta = 1}^{b-1} b^{-2c} = \frac{1}{b+1}$ and  $a_1 > a_2$. Thus the result follows.
\end{proof}

Let now $f(x) = f_0 + f_1 x + f_2 x^2 + \cdots$. The $k$th Walsh
coefficient of $f$ is given by
\begin{eqnarray*}
\hat{f}(k) & = & \int_0^1 f(x) \overline{\wal_k(x)} \rd x \\ & = &
\sum_{r=0}^\infty f_r \int_0^1 x^r \overline{\wal_k(x)} \rd x \\ & =
& \sum_{r = v}^\infty f_r \chi_{r,v}(a_1,\ldots, a_v;\kappa_1,\ldots,
\kappa_v).
\end{eqnarray*}
We can estimate the $k$th Walsh coefficient by
\begin{eqnarray*}
|\hat{f}(k)| & = & \left|\sum_{r=v}^\infty \chi_{r,v}(a_1,\ldots,
a_v;\kappa_1,\ldots, \kappa_v) f_r \right| \\ &
\le & \sum_{r=v}^\infty |\chi_{r,v}(a_1,\ldots,
a_v;\kappa_1,\ldots,\kappa_v)| |f_r| \\ & \le &
\sum_{r=v}^\infty |f_r| \min_{0 \le u \le v} b^{-\mu_u(k)} \frac{r!}{(r-u+1)!}
\frac{3^{\min(1,u)}}{(2\sin\frac{\pi}{b})^u} \\ && \qquad\qquad\qquad\qquad \times \left(1 + \frac{1}{b} + \frac{1}{b(b+1)}\right)^{\max(0,u-1)} \\
& \le  & \min_{0 \le u \le v} b^{-\mu_u(k)}
\frac{3^{\min(1,u)}}{(2\sin\frac{\pi}{b})^u} \left(1 + \frac{1}{b} + \frac{1}{b(b+1)}\right)^{\max(0,u-1)} \\ && \qquad\qquad\qquad\qquad \times \sum_{r=v}^\infty
\frac{r! |f_r|}{(r-u+1)!}.
\end{eqnarray*}

Hence we have shown the following theorem.
\begin{theorem}\label{thm_deltainfty}
Let $f(x) = f_0 + f_1 x + f_2 x^2 + \cdots$ and let $k \in \NN$. Then
we have
\begin{equation*}
|\hat{f}(k)| \le \min_{0 \le u \le v} b^{-\mu_u(k)}
\frac{3^{\min(1,u)}}{(2\sin\frac{\pi}{b})^u} \left(1+\frac{1}{b} + \frac{1}{b(b+1)}\right)^{\max(0,u-1)} \sum_{r=v}^\infty
\frac{r! |f_r|}{(r-u+1)!}.
\end{equation*}
\end{theorem}

\begin{remark}
This result cannot be directly obtained from \cite{Dick06}, as there
the constant for a power series would be infinite.
\end{remark}

The bound in the theorem makes of course only sense for $u$ for
which $\sum_{r=v}^\infty \frac{r! |f_r|}{(r-u+1)!}$ is finite. We
give some examples:
\begin{itemize}
\item For $f \in C^\infty([0,1])$ we have $f^{(r)}(0) = r! f_r$. If $|f^{(r)}(0)|$ grows exponentially
(like for $f(x) = \de^{a x}$ with $a > 1$), then $\sum_{r=v}^\infty \frac{|f^{(r)}(0)|}{(r-v+1)!}$ will
be finite for any $v \in \NN$. The theorem implies that the Walsh coefficients decay with order
$\mathcal{O}(b^{-\mu_v(k)})$.

\item Using Sterling's formula we obtain that $\frac{r!}{(r-v+1)!} \approx (r-v+1)^{v-1}$ as $r$ tends to $\infty$.
For $f(x) = \frac{1}{1-cx}$ with $0 < c < 1$ we have $f_r = c^r$.
In this case we have $$\sum_{r=v}^\infty \frac{r! |f_r|}{(r-v+1)!}
\approx \sum_{r=v}^\infty (r-v+1)^{v-1} c^r = c^{v-1}
\sum_{r=1}^\infty r^{v-1} c^r < \infty,$$ for all $v \in \NN$. The
theorem implies that the Walsh coefficients decay with order
$\mathcal{O}(b^{-\mu_v(k)})$.
\end{itemize}

For $f \in C^\infty([0,1])$ with $f(x) = \sum_{r=0}^\infty f_r x^r$
we define the semi-norm
\begin{equation*}
\|f\| = \sum_{r=1}^\infty |f_r| = \sum_{r = 1}^\infty
\frac{|f^{(r)}(0)|}{r!}.
\end{equation*}

Then the $(v-1)$th derivative of $f$ is given by $$f^{(v-1)}(x) =
\sum_{r=0}^\infty  \frac{(v-1+r)!}{r!} f_{v-1+r} x^r = \sum_{r =
v-1}^\infty \frac{r!}{(r-v+1)!} f_r x^{r-v+1}$$ and
\begin{equation*}
\|f^{(v-1)}\| = \sum_{r=v}^\infty \frac{r! |f_r|}{(r-v+1)!} =
\sum_{r=v}^\infty \frac{|f^{(r)}(0)|}{(r-v+1)!}.
\end{equation*}

Hence we obtain the following corollary from
Theorem~\ref{thm_deltainfty}.
\begin{corollary}
Let $f \in C^\infty([0,1])$ with $\|f^{(z)}\| < \infty$ for all $z
\in \NN_0$. Then for every $k \in \NN$ we have
\begin{equation*}
|\hat{f}(k)| \le b^{-\mu_v(k)}
\frac{3}{(2\sin\frac{\pi}{b})^v} \left(1+\frac{1}{b} + \frac{1}{b(b+1)}\right)^{v-1}  \|f^{(v-1)}\|.
\end{equation*}
\end{corollary}

Let us consider another example: Let $f_r = r^{-\delta}$, with
$\delta > 1$. So, for example, we can choose $u = \min(v, \lceil
\delta \rceil -2)$ in the theorem above, which will guarantee that
$\sum_{r = v}^\infty \frac{r! |f_r|}{(r-u+1)!} < \infty$. On the
other hand, this sum is not finite for $\lceil \delta \rceil-2 < u
\le v$. The theorem implies that the Walsh coefficients decay with
order $\mathcal{O}(b^{-\mu_{\min(v, \lceil\delta \rceil - 2)}(k)})$.
Note that this function $f$ is only $\lceil \delta \rceil - 2$ times
continuously differentiable. We will consider this case in the next
section.

\section{On the Walsh coefficients of functions in $C^r([0,1])$}

In this section we prove an explicit constant $C_{r}$, which is better than the constant
which can be obtained from \cite{Dick06}.

Before the next lemma we introduce a variation of fractional order:
For $0 < \lambda \le 1$ and $f:[0,1]\rightarrow \real$ let
\begin{equation*}
V_\lambda(f) = \sup_{0 = x_0 < x_1 < \ldots < x_{N-1} < x_N =  1} \sum_{n=1}^N |x_{n}- x_{n-1}| \frac{|f(x_n) - f(x_{n-1})|}{|x_n-x_{n-1}|^\lambda},
\end{equation*}
where the supremum is taken over all partitions of the interval $[0,1]$.

If $f$ has a continuous first derivative on $[0,1]$, then
\begin{equation*}
V_1(f) = \int_0^1 |f'(x)| \rd x.
\end{equation*}
If $f$ satisfies a H\"older condition of order $0 < \lambda \le 1$, i.e., $|f(x)-f(y)| \le C_f |x-y|^\lambda$ for all $x,y \in [0,1]$, then $V_\lambda(f) \le C_f$.

The following lemma appeared already in \cite{Fine} (albeit in a slightly different form, see also \cite{Dick06,pirsic}).
\begin{lemma}\label{lem_boundVf}
Let $0 < \lambda \le 1$ and let $f \in \LL_2([0,1])$ satisfy
$V_\lambda(f) < \infty$. Then for any $k \in \NN$, the $k$th Walsh coefficient satisfies
\begin{equation*}
|\hat{f}(k)| \le (b-1)^{1+\lambda} b^{-\lambda a_1} V_\lambda(f).
\end{equation*}
\end{lemma}

Thus, the decay of the Walsh coefficients of functions with smoothness $0 < r \le 1$ has already been considered and we deal with $r > 1$ in the following.

Let now $f \in \LL_2([0,1])$ with $V_\lambda(f) < \infty$ and let
$F_1(x) = \int_0^x f(y) \rd y$. Then using integration by parts as
in the previous section, we obtain for $k \in\NN$
\begin{equation*}
\hat{F_1}(k) = \int_0^1 F_1(x) \overline{\wal_k(x)} \rd x = -
\int_0^1 f(x) J_k(x) \rd x.
\end{equation*}
Substituting the Walsh series for
$J_k$ from Lemma~\ref{lem_Jk}, we obtain
\begin{eqnarray*}
\hat{F_1}(k) & = & -b^{-a_1}\bigg((1-\omega_b^{-\kappa_1})^{-1}
\hat{f}(k') + (1/2+(\omega_b^{-\kappa_1}-1)^{-1}) \hat{f}(k) \\ &&
\qquad\quad + \sum_{c=1}^\infty \sum_{\vartheta=1}^\infty b^{-c}
(\omega_b^\vartheta-1)^{-1} \hat{f}(\vartheta b^{a_1+c-1} + k) \bigg).
\end{eqnarray*}
Taking the absolute value on both sides and using the same estimations as in the previous section, we obtain
\begin{equation}\label{ineq_boundhatF}
|\hat{F_1}(k)| \le \frac{b^{-a_1}}{2\sin\frac{\pi}{b}} (
|\hat{f}(k')| + |\hat{f}(k)|  + (b-1) \sum_{c=1}^\infty b^{-c}
|\hat{f}(\vartheta b^{a_1+c-1} + k)|).
\end{equation}

Thus, using Lemma~\ref{lem_boundVf}  we obtain for $k \in\NN$ with $v \ge 2$, that
\begin{equation*}
|\hat{F_1}(k)| \le b^{-a_1 - \lambda a_2} V_\lambda(f)
\frac{(b-1)^{1+\lambda}}{2 \sin\frac{\pi}{b}} (1+2 b^{-\lambda}).
\end{equation*}

For $k = \kappa_1 b^{a_1-1}$ we obtain
\begin{equation*}
|\hat{F_1}(k)| \le  \frac{b^{-a_1}}{2
\sin\frac{\pi}{b}} \left(|\hat{f}(0)| + 2(b-1)^{1+\lambda} b^{-\lambda a_1} V_\lambda(f) \right).
\end{equation*}

Defining $F_r(x) = \int_0^x F_{r-1}(y) \rd y$ for $r \ge 1$, we can
obtain bounds on the Walsh coefficients of $F_r$ by using induction
on $r$. Using similar arguments as in the proof of Lemma~\ref{lem_bound_Walsh_xr} we obtain for $v > r$ that
\begin{equation}\label{eq_walshcoeffvr}
|\hat{F_r}(k)| \le b^{-\mu_r(k) -\lambda a_{r+1}} V_\lambda(f)
\frac{(b-1)^{1+\lambda} (1+2 b^{-\lambda})}{(2 \sin\frac{\pi}{b})^r} \left(1 + \frac{1}{b} + \frac{1}{b(b+1)}\right)^{r-1},
\end{equation}
and for $ v = r$ that
\begin{eqnarray}\label{eq_walshcoeffveqr}
\lefteqn{ |\hat{F_r}(k)| } \\ & \le & \frac{b^{-\mu_r(k)}}{(2\sin\frac{\pi}{b})^r} \left(1+\frac{1}{b} + \frac{1}{b(b+1)}\right)^{r-1}  \left(|\hat{f}(0)| + 2(b-1)^{1+\lambda} b^{-\lambda a_r} V_\lambda(f) \right). \nonumber
\end{eqnarray}
For $1 \le v < r$ we have
\begin{eqnarray}\label{eq_walshcoeffrv}
|\hat{F_r}(k)| & \le &  \frac{b^{-\mu_r(k)}}{(2\sin\frac{\pi}{b})^v} \left(1+\frac{1}{b} +\frac{1}{b(b+1)}\right)^{v-1} \nonumber \\ && \times \left(|\hat{F}_{r-v}(0)| + 2 (b-1)^{1+\lambda} b^{-\lambda a_v} V_\lambda(F_{r-v}) \right).
\end{eqnarray}

Note that we also have $F_r(x) = \int_0^1 f(t)
\frac{(x-t)_+^{r-1}}{(r-1)!} \rd t$, where $(x-t)_+^{r-1} =
(x-t)^{r-1} 1_{[0,x)}(t)$ for $0 \le x,t \le 1$ and $1_{[0,x)}(t)$
is $1$ for $t \in [0,x)$ and $0$ otherwise.

A function $f \in C^r([0,1])$ for which
$V_\lambda(f^{(r)}) < \infty$ can be represented by a Taylor series
\begin{equation*}
f(x) = f(0) + \frac{f^{(1)}(0)}{1!} x + \cdots + \frac{f^{(r-1)}(0)}{(r-1)!} x^{r -1} + \int_0^1 f^{(r)}(t) \frac{(x-t)_+^{r -1}}{(r-1)!} \rd t.
\end{equation*}

With this we can now obtain a bound on the Walsh coefficients of $f$. For $v \ge r$ we know from \cite{Dick06} that $$\int_0^1 \left(f(0) + \frac{f^{(1)}(0)}{1!} x + \cdots + \frac{f^{(r-1)}(0)}{(r-1)!} x^{r -1} \right) \overline{\wal_k(x)} \rd x = 0.$$ To bound the Walsh coefficient of $\int_0^1 f^{(r)}(t) \frac{(x-t)_+^{r -1}}{(r-1)!} \rd t$ for $v > r$ we can use (\ref{eq_walshcoeffvr}) to obtain
\begin{equation*}
|\hat{f}(k)| \le  b^{-\mu_r(k) -\lambda a_{r+1}} V_\lambda(f^{(r)})
\frac{(b-1)^{1+\lambda} (1+2 b^{-\lambda})}{(2 \sin\frac{\pi}{b})^r} \left(1+\frac{1}{b} + \frac{1}{b(b+1)}\right)^{r-1}.
\end{equation*}
For $v = r$ we can use (\ref{eq_walshcoeffveqr}) to obtain
\begin{eqnarray*}
|\hat{f}(k)|  & \le & \frac{b^{-\mu_r(k)}}{(2\sin\frac{\pi}{b})^r} \left(1 +\frac{1}{b} + \frac{1}{b(b+1)}\right)^{r-1}  \\ && \qquad\qquad\qquad \times \left(|\hat{f}^{(r)}(0)| + 2(b-1)^{1+\lambda} b^{-\lambda a_r} V_\lambda(f^{(r)}) \right).
\end{eqnarray*}
For $1 \le v < r$ we have
\begin{eqnarray*}
\lefteqn{\left| \int_0^1  \left(f(0) + \frac{f^{(1)}(0)}{1!} x + \cdots + \frac{f^{(r-1)}(0)}{(r-1)!} x^{r -1} \right) \overline{\wal_k(x)} \rd x \right|} \qquad \\  & \le &  b^{-\mu_r(k)}
\frac{3}{(2\sin\frac{\pi}{b})^v} \left(1 + \frac{1}{b} + \frac{1}{b(b+1)} \right)^{v-1} \sum_{s=v}^{r-1}
\frac{|f^{(s)}(0)|}{(s-v+1)!}
\end{eqnarray*}
and therefore, using (\ref{eq_walshcoeffrv}), we obtain
\begin{eqnarray*}
|\hat{f}(k)| & \le & \frac{b^{-\mu_r(k)}}{(2\sin\frac{\pi}{b})^v} \left(1+\frac{1}{b} + \frac{1}{b(b+1)}\right)^{v-1} \\ && \times \left[3 \sum_{s=v}^{r-1} \frac{|f^{(s)}(0)|}{(s-v+1)!}  +    \left(|\hat{f}^{(v)}(0)| + 2 (b-1)^{1+\lambda} b^{-\lambda a_v} V_\lambda(f^{(v)}) \right) \right],
\end{eqnarray*}
where $\hat{f}^{(v)}(0)$ denotes the $0$th Walsh coefficient of $f^{(v)}$. We have shown the following theorem.

\begin{theorem}\label{thm_cr}
Let $f \in C^r([0,1])$ with $V_\lambda(f^{(r)}) < \infty$, and let $k \in\NN$. Then for $v > r$ we have
\begin{equation*}
|\hat{f}(k)| \le  b^{-\mu_r(k) -\lambda a_{r+1}} V_\lambda(f^{(r)})
\frac{(b-1)^{1+\lambda} (1+2 b^{-\lambda})}{(2 \sin\frac{\pi}{b})^r} \left(1 + \frac{1}{b} + \frac{1}{b(b+1)}\right)^{r-1},
\end{equation*}
for $r = v$ we have
\begin{eqnarray*}
|\hat{f}(k)| & \le & \frac{b^{-\mu_r(k)}}{(2\sin\frac{\pi}{b})^r} \left(1 + \frac{1}{b} + \frac{1}{b(b+1)}\right)^{r-1} \\ && \times \left(\left|\int_0^1 f^{(r)}(x) \rd x \right| + 2(b-1)^{1+\lambda} b^{-\lambda a_r} V_\lambda(f^{(r)}) \right),
\end{eqnarray*}
and for $v < r$ we have
\begin{eqnarray*}
\lefteqn{ |\hat{f}(k)| } \\ & \le & \frac{b^{-\mu_r(k)}}{(2\sin\frac{\pi}{b})^v} \left(1 + \frac{1}{b} + \frac{1}{b(b+1)}\right)^{v-1} \\ && \times \left[3 \sum_{s=v}^{r-1} \frac{|f^{(s)}(0)|}{(s-v+1)!} +   \left(\left|\int_0^1 f^{(v)}(x) \rd x\right| + 2 (b-1)^{1+\lambda} b^{-\lambda a_v} V_\lambda(f^{(v)}) \right) \right].
\end{eqnarray*}
\end{theorem}

We also prove bounds on the decay of the Walsh coefficients of functions from Sobolev spaces. For this, we first need bounds on the Walsh coefficients of Bernoulli polynomials, which we consider in the next section.

\section{On the Walsh coefficients of Bernoulli polynomials}

For $r \ge 0$ let $B_r(\cdot)$ denote the Bernoulli polynomial of
degree $r$ and $b_r(\cdot) = \frac{B_r(\cdot)}{r!}$. For example we
have $B_0(x) = 1$, $B_1(x) = x- 1/2$, $B_2(x) = x^2-x+1/6$ and so
on. Those polynomials have the properties
\begin{equation*}
b'_r(x) = b_{r-1}(x) \quad \mbox{and} \quad \int_0^1 b_r(x) = 0
\quad \mbox{for all } r \ge 1.
\end{equation*}
We obviously have $b_0'(x) = 0$ and $\int_0^1 b_0(x) \rd x = 1$.
Further, $B_r(1-x) = (-1)^r B_r(x)$ and also $b_r(1-x) = (-1)^r
b_r(x)$. The numbers $B_r = B_r(0)$ are the Bernoulli numbers and
$B_r = 0$ for all odd $r \ge 3$. Further, for $r \ge 2$, we have
\begin{equation}\label{eq_Fourierseriesber}
b_r(x) = -\frac{1}{(2\pi\icomp)^r} \sum_{h \in \integer\setminus
\{0\}} h^{-r} \de^{2\pi\icomp h x}, \qquad \mbox{for } 0 \le x \le 1.
\end{equation}

It is more convenient to calculate with $b_r(\cdot)$ rather than the
Bernoulli polynomials.

For $r \ge 1$ and $k \in \NN$ let
\begin{equation*}
\beta_{r,v}(a_1,\ldots, a_v;\kappa_1,\ldots, \kappa_v) = \int_0^1 b_r(x)
\overline{\wal_k(x)} \rd x.
\end{equation*}
As for $\chi_{r,v}$, we also have $\beta_{r,v} = 0$ for $v > r$. Further,
for $k = 0$ let $v=0$ and we have $\beta_{r,0} = 0$ for all $r \ge 1$.

The Walsh series for $b_1$ can be obtained from the Walsh series of
$J_0$ from Lemma~\ref{lem_Jk} and is given by
\begin{equation*}
b_1(x) = x-1/2 = \sum_{c=1}^\infty \sum_{\vartheta=1}^{b-1} b^{-c}
(\omega_b^{-\vartheta}-1)^{-1} \wal_{\vartheta b^{c-1}}(x).
\end{equation*}
Thus $$\beta_{1,1}(a_1;\kappa_1) = -b^{-a_1}
(1-\omega_b^{-\kappa_1})^{-1}.$$

Using integration by parts and $J_k(0) = J_k(1) = 0$ we obtain for
all $r > 1$ that
\begin{equation}\label{eq_bernoulliJk}
\int_0^1 b_r(x) \overline{\wal_k(x)}\rd x = - \int_0^1 b_{r-1}(x)
J_k(x) \rd x.
\end{equation}

Using Lemma~\ref{lem_Jk} and (\ref{eq_bernoulliJk}) we obtain for $1
\le v \le r$ and $r > 1$ that
\begin{eqnarray}\label{eq_recursion_bernoulli}
\lefteqn{ \beta_{r,v}(a_1,\ldots, a_v; \kappa_1,\ldots, \kappa_v) } \\ &
= & -b^{-a_1} \bigg((1-\omega_b^{-\kappa_1})^{-1}
\beta_{r-1,v-1}(a_2,\ldots,a_v;\kappa_2,\ldots, \kappa_v) \nonumber \\
&& \qquad\qquad + (1/2+(\omega_b^{-\kappa_1}-1)^{-1}) \beta_{r-1,v}(a_1,
\ldots, a_v;\kappa_1, \ldots, \kappa_v) \nonumber \\ && \qquad\qquad
+ \sum_{c=1}^\infty \sum_{\vartheta=1}^{b-1} b^{-c} (\omega_b^\vartheta-1)^{-1}
\beta_{r-1,v+1}(a_1+c, a_1,\ldots,a_v; \vartheta,
\kappa_1,\ldots,\kappa_v)\bigg). \nonumber
\end{eqnarray}

From (\ref{eq_recursion_bernoulli}) we can obtain
\begin{equation*}
\beta_{r,r}(a_1,\ldots, a_r;\kappa_1,\ldots, \kappa_r) = (-1)^r
b^{-a_1-\cdots - a_r} \prod_{s=1}^r (1-\omega_b^{-\kappa_s})^{-1}
\end{equation*}
for all $r \ge 1$.

The first few values of $\beta_{r,v}$ are as follows:
\begin{itemize}
\item $r=1$: $\beta_{1,0} = 0$, $\beta_{1,1}(a_1;\kappa_1) = - b^{-a_1} (1-\omega_b^{-\kappa_1})^{-1}$;
\item $r = 2$: $\beta_{2,0} = 0$, $\beta_{2,1}(a_1;\kappa_1) = b^{-2a_1} (1-\omega_b^{-\kappa_1})^{-1} (1/2 + (\omega_b^{-\kappa_1}-1)^{-1})$, $\beta_{2,2}(a_1,a_2;\kappa_1,\kappa_2) = b^{-a_1-a_2} (1-\omega_b^{-\kappa_1})^{-1} (1-\omega_b^{-\kappa_2})^{-1}$;
\end{itemize}

In principle we can obtain all values of $\beta_{r,v}$ recursively
using (\ref{eq_recursion_bernoulli}). We calculated already
$\beta_{r,v}$ for $v = r$
 and we could continue doing so for $v = r-1, \ldots,
1$. But the formulae become increasingly complex, so we only prove a
bound on them.

For any $r \ge 0$ and a non-negative integer $k$ we introduce the function
\begin{equation*}
\mu_{r, {\rm per}}(k) = \left\{\begin{array}{ll} 0 & \mbox{for } r = 0, k \ge 0, \\  0 & \mbox{for } k =  0, r \ge 0, \\ a_1 + \cdots + a_v + (r-v) a_v & \mbox{for } 1 \le v < r, \\ a_1 + \cdots + a_r & \mbox{for } v \ge r. \end{array} \right.
\end{equation*}

\begin{lemma}\label{lem_bound_walsh_bernoulli}
For any $r \ge 2$ and $1 \le v \le r$ we have
\begin{equation*}
|\beta_{r,v}(a_1,\ldots, a_v;\kappa_1,\ldots, \kappa_v)| \le \frac{b^{-\mu_{r, {\rm per}}(k)}}{(2\sin\frac{\pi}{b})^r} \left(1+\frac{1}{b}+\frac{1}{b(b+1)}\right)^{r-2}.
\end{equation*}
\end{lemma}

\begin{proof}
We prove the bound by induction on $r$. Using Lemma~\ref{lem_elementary} it can easily be seen that the result holds for $r = 2$. Hence assume now that $r > 2$ and the result holds for $r -1$. By taking the absolute value of (\ref{eq_recursion_bernoulli}) and using the triangular inequality together with  Lemma~\ref{lem_elementary} we obtain
\begin{eqnarray*}
\lefteqn{ |\beta_{r,v}(a_1,\ldots, a_v; \kappa_1,\ldots, \kappa_v)| } \\ &
\le & \frac{b^{-a_1}}{2\sin\frac{\pi}{b}} \bigg(
|\beta_{r-1,v-1}(a_2,\ldots,a_v;\kappa_2,\ldots, \kappa_v)| + |\beta_{r-1,v}(a_1,
\ldots, a_v;\kappa_1, \ldots, \kappa_v)| \nonumber \\ && \qquad\qquad
+ \sum_{c=1}^\infty \sum_{\vartheta=1}^{b-1}  b^{-c} |\beta_{r-1,v+1}(a_1+c, a_1,\ldots,a_v; \vartheta,
\kappa_1,\ldots,\kappa_v)|\bigg).
\end{eqnarray*}
We can now use the induction assumption for $|\beta_{r-1,v-1}|, |\beta_{r-1,v}|, |\beta_{r-1,v+1}|$.  Hence, for $v > 1$, we obtain
\begin{eqnarray*}
|\beta_{r,v}(a_1,\ldots, a_v; \kappa_1,\ldots, \kappa_v)|  &
\le & \frac{b^{-\mu_{r, {\rm per}} (k)}}{(2\sin\frac{\pi}{b})^r} \left(1 + \frac{1}{b} + \frac{1}{b(b+1)}\right)^{r-3} \\ && \times  \left(1 + b^{a_2-a_1} + \sum_{c=1}^\infty \sum_{\vartheta=1}^{b-1}  b^{-2 c} b^{a_2-a_1} \right).
\end{eqnarray*}
By noting that $\sum_{c=1}^\infty \sum_{\vartheta=1}^{b-1} b^{-2c} = \frac{1}{b+1}$, and $a_1 > a_2$ we obtain the result.

For $v = 1$ note that $\beta_{r,0} = 0$. In this case we have
\begin{eqnarray*}
|\beta_{r,1}(a_1; \kappa_1)|  &
\le & \frac{b^{-\mu_{r, {\rm per}} (k)}}{(2\sin\frac{\pi}{b})^r} \left(1 + \frac{1}{b} + \frac{1}{b(b+1)}\right)^{r-3}  \left(1 + \sum_{c=1}^\infty \sum_{\vartheta=1}^{b-1}  b^{-2 c} \right) \\ & \le &  \frac{b^{-\mu_{r, {\rm per}} (k)}}{(2\sin\frac{\pi}{b})^r} \left(1 + \frac{1}{b} + \frac{1}{b(b+1)}\right)^{r-2},
\end{eqnarray*}
which implies the result.
\end{proof}

The $b_r$ are polynomials, but using (\ref{eq_Fourierseriesber}) we can extend $b_r$ periodically so that it is defined on $\real$. We denote those functions by $\widetilde{b}_r$. Then for $r \ge 1$ we have
\begin{equation*}
\widetilde{b}_{2r}(x) = \frac{2 (-1)^{r+1}}{(2\pi)^{2r}} \sum_{h=1}^\infty h^{-2r} \cos 2 \pi h x \qquad \mbox{for } x \in \real,
\end{equation*}
and
\begin{equation*}
\widetilde{b}_{2r+1}(x) = \frac{2 (-1)^{r+1}}{(2\pi)^{2r+1}} \sum_{h=1}^\infty h^{-2r-1} \sin 2 \pi h x \qquad \mbox{for } x \in \real.
\end{equation*}
From this it can be seen that $\widetilde{b}_{r}(x) = (-1)^r \widetilde{b}_{r}(-x)$ for all $r \ge 2$. Note that for $x,y \in [0,1]$ we have $b_{2r}(|x-y|) = \widetilde{b}_{2r}(x-y)$ and $b_{2r+1}(|x-y|) = (-1)^{1_{x < y}} \widetilde{b}_{2r+1}(x-y)$, where $1_{x < y}$ is $1$ for $x < y$ and $0$ otherwise. We also extend $B_r(\cdot)$ periodically to $\real$, which we denote by $\widetilde{B}_r(\cdot)$.

In the next section we will also need a bound on the Walsh coefficients of $\widetilde{b}_r(x-y)$. For $k,l \ge 0$ let
\begin{eqnarray}\label{def_gammakl}
\gamma_{r}(k,l)  & = & \int_0^1 \int_0^1 \widetilde{b}_r(x-y) \overline{\wal_k(x)} \wal_l(y) \rd x \rd y  \\ & = &
-\frac{1}{(2\pi\icomp)^r} \sum_{h \in \integer\setminus\{0\}} h^{-r} \tau_{h,k} \overline{\tau_{h,l}}, \nonumber
\end{eqnarray}
where
\begin{equation*}
\tau_{h,k} = \int_0^1 \de^{2\pi \icomp h x} \overline{\wal_k(x)} \rd x.
\end{equation*}
We have $\gamma_r(k,0) = \gamma_r(0,l) = 0$ for all $k,l \ge 0$, as $\int_{z}^{1+z} \widetilde{b}_r(x) \rd x = 0$ for any $z \in \real$. Further we have $\gamma_r(l,k) = (-1)^r \overline{\gamma_r(k,l)}$ and therefore also $|\gamma_r(k,l)| = |\gamma_r(l,k)|$.

We obtain bounds on $\gamma_r$ by induction. In the next lemma we calculate the values of $\gamma_2$.

\begin{lemma}\label{lem_eq_gamma2}
For all $k, l \ge 0$ we have $\gamma_2(k,0) = \gamma_2(0,l) = 0$. For $k,l > 0$ we have
\begin{equation*}
\gamma_2(k,l) = \left\{\begin{array}{ll}  b^{-2a_1} \left(\frac{1}{2 \sin^2 \kappa_1 \pi/b} - \frac{1}{3} \right) & \mbox{if } k = l, \\   b^{-a_1-d_1} (\omega_b^{-\kappa_1}-1)^{-1} (\omega_b^{\lambda_1}-1)^{-1} & \mbox{if } k' = l' > 0, \\ & \mbox{and } k \neq l, \\ b^{-a_1-d_1} (1/2+(\omega_b^{-\lambda_1}-1)^{-1})(\omega_b^{-\kappa_1}-1)^{-1} & \\ + b^{-2a_1} (1/2+(\omega_b^{\kappa_1}-1)^{-1})(1-\omega_b^{-\kappa_1})^{-1} & \mbox{if } k' = l, \\  b^{-a_1-d_1} (1/2+(\omega_b^{\kappa_1}-1)^{-1})(\omega_b^{\lambda_1}-1)^{-1} & \\ + b^{-2d_1} (1/2+(\omega_b^{-\lambda_1}-1)^{-1})(1-\omega_b^{\lambda_1})^{-1} & \mbox{if } k = l', \\ b^{-a_1-a_2} (1-\omega_b^{-\kappa_2})^{-1} (\omega_b^{-\kappa_1}-1)^{-1} & \mbox{if } k'' =  l,  \\ b^{-d_1-d_2} (1-\omega_b^{\lambda_2})^{-1} (\omega_b^{\lambda_1}-1)^{-1} & \mbox{if } k = l'',  \\ 0 & \mbox{otherwise}.  \end{array}\right.
\end{equation*}
\end{lemma}

\begin{proof}
Note that $\gamma_2(k,0) = \gamma_2(0,l) = 0$ for all $k,l \ge 0$, as $\int_z^{1+z} \widetilde{b}_2(x) \rd x = 0$ for any $z \in \real$.

Now assume that $k,l > 0$. The value of $\gamma_2(k,k)$ has been obtained in \cite[Appendix A]{DP} (but can also be obtained from the following).

The Walsh series for $\widetilde{b}_2(x-y) = b_2(|x-y|) =
\frac{(x-y)^2}{2} - \frac{|x-y|}{2} + \frac{1}{6}$ can be calculated
in the following way: We have $x = \overline{J_0(x)}$ and $y =
J_0(y)$ and so
\begin{equation*}
\frac{(x-y)^2}{2} = \frac{(\overline{J_0(x)} - J_0(y))^2}{2}.
\end{equation*}

Further
\begin{equation*}
|x-y| = x + y - 2 \min(x,y) = x + y - 2 \int_0^1 1_{[0,x)}(t) 1_{[0,y)}(t) \rd t,
\end{equation*}
where $1_{[0,x)}(t)$ is $1$ for $t \in [0,x)$ and $0$ otherwise. Note that $J_k(x) = \int_0^x \overline{\wal_k(t)} \rd t = \int_0^1 1_{[0,x)}(t) \overline{\wal_k(t)} \rd t$, which implies $$1_{[0,x)}(t) = \sum_{k=0}^\infty J_k(x) \wal_k(t).$$
Thus
\begin{eqnarray*}
\min(x,y) & = & \int_0^1 1_{[0,x)}(t) 1_{[0,y)}(t) \rd t \\ & = & \sum_{m,n=0}^\infty \overline{J_m(x)} J_n(y) \int_0^1 \overline{\wal_m(t)} \wal_n(t) \rd t \\ & = & \sum_{m=0}^\infty \overline{J_m(x)} J_m(y).
\end{eqnarray*}

The Walsh series for $\widetilde{b}_2(x-y)$ is therefore given by
\begin{equation*}
\widetilde{b}_2(x-y) =  \frac{(\overline{J_0(x)})^2 + (J_0(y))^2 - \overline{J_0(x)} - J_0(y)}{2} + \sum_{m=1}^\infty \overline{J_m(x)} J_m(y) + \frac{1}{6}.
\end{equation*}

 We have
\begin{eqnarray*}
\lefteqn{ \gamma_2(k,l) } \\ & = & \int_0^1 \int_0^1
\widetilde{b}_2(x-y) \overline{\wal_k(x)} \wal_l(y) \rd x \rd y \\ &
= & \int_0^1 \int_0^1 \left[\frac{(\overline{J_0(x)})^2 + (J_0(y))^2
- \overline{J_0(x)} - J_0(y)}{2} + \sum_{m=1}^\infty
\overline{J_m(x)} J_m(y) + \frac{1}{6} \right] \\ && \qquad\qquad
\times \overline{\wal_k(x)} \wal_l(y) \rd x \rd y \\ & = &
\sum_{m=1}^\infty \int_0^1 \overline{J_m(x)}  \overline{\wal_k(x)}
\rd x \int_0^1 J_m(y) \wal_l(y) \rd y.
\end{eqnarray*}
It remains to consider the integral $\int_0^1 \overline{J_m(x)} \overline{\wal_k(x)} \rd x$. Let $m = \eta b^{e-1} + m'$, with $0 < \eta < b$, $e > 0$, and $0 \le m' < b^{e-1}$. Then we have
\begin{eqnarray*}
\lefteqn{\int_0^1 \overline{J_m(x)} \overline{\wal_k(x)} \rd x } \\ & = & b^{-e} \bigg (1-\omega_b^{\eta})^{-1} \int_0^1 \wal_{m'}(x) \overline{\wal_k(x)} \rd x \\ && \qquad + (1/2 + (\omega_b^{\eta}-1)^{-1}) \int_0^1 \wal_m(x) \overline{\wal_k(x)} \rd x \\ && \qquad + \sum_{c=1}^\infty \sum_{\vartheta=1}^{b-1} b^{-c} (\omega_b^{-\vartheta}-1)^{-1} \int_0^1 \wal_{\vartheta b^{e+c-1} + m}(x) \overline{\wal_k(x)} \rd x\bigg).
\end{eqnarray*}
This integral is not $0$ only if either $m' = k$, $m = k$ or $m + \vartheta b^{e+c-1} = k$ for some $\vartheta,c$. Analogously the same applies to the integral $\int_0^1 J_m(y) \wal_l(y) \rd y$. Hence we only need to consider a few cases for which $\gamma_2(k,l)$ is non-zero, and by going through each of them we obtain the result.
\end{proof}

Note that many values for $\gamma_2(k,l)$ are $0$, in particular, if $k$ and $l$ are sufficiently `different' from each other. This property is inherited by $b_r$ for $r > 2$ via the recursion
\begin{eqnarray}\label{eq_recursion_gamma}
\gamma_r(k,l) & = & - b^{-a_1} \bigg((1-\omega_b^{-\kappa_1})^{-1} \gamma_{r-1}(k',l) + (1/2+(\omega_b^{-\kappa}-1)^{-1}) \gamma_{r-1}(k,l) \nonumber \\ && \qquad\quad + \sum_{c=1}^\infty \sum_{\vartheta=1}^{b-1} b^{-c} (\omega_b^{\vartheta} - 1)^{-1} \gamma_{r-1}(\vartheta b^{c+a_1-1} + k, l)\bigg).
\end{eqnarray}
This recursion is obtained from
\begin{equation}\label{eq_intparts_bx}
\gamma_r(k,l)  =   - \int_0^1 \int_0^1 \widetilde{b}_{r-1}(x-y) J_k(x) \wal_l(y) \rd x \rd y,
\end{equation}
which in turn can be obtained using integration by parts. In the following lemma we show that $\gamma_r(k,l) = 0$ for many choices of $k$ and $l$.

\begin{lemma}\label{lem_gamma_0}
\begin{itemize}
\item[1.)] For any $k,l \ge 0$ we have $\gamma_r(k,0) = \gamma_r(0,l) = 0$.
\item[2.)] For $k,l > 0$ with $|v-w| > r$ we have $\gamma_r(k,l) = 0$.
\item[3.)] Let $k, l > 0$ such that $|v-w| \le r$.
\begin{itemize}
\item[(i)] If $v = 1$, but $(\kappa_1, a_1) \neq (\lambda_w, d_w)$, then $\gamma_r(k,l) = 0$.
\item[(ii)] If $w = 1$, but $(\lambda_1, d_1) \neq (\kappa_v, a_v)$, then $\gamma_r(k,l) = 0$.
\item[(iii)] If $r-1 \le |v-w| \le r$, but
\begin{eqnarray*}
\lefteqn{ (a_{v-\min(v,w)+1}, \ldots, a_v,  \kappa_{v-\min(v,w)+1}, \ldots, \kappa_v) } \\ & \neq & (d_{w-\min(v,w)+1}, \ldots, d_w, \lambda_{w-\min(v,w)+1}, \ldots, \lambda_w),
\end{eqnarray*}
then $\gamma_r(k,l) = 0$.
\item[(iv)] If $v, w > 1$ and $0 \le |v-w| \le r-2$, but
\begin{eqnarray*}
\lefteqn{ (a_{v-\min(v,w)+2}, \ldots, a_v,  \kappa_{v-\min(v,w)+2}, \ldots, \kappa_v) } \\ & \neq & (d_{w-\min(v,w)+2}, \ldots, d_w, \lambda_{w-\min(v,w)+2}, \ldots, \lambda_w),
\end{eqnarray*}
then $\gamma_r(k,l) = 0$.
\end{itemize}
\end{itemize}
\end{lemma}

\begin{proof}
\begin{itemize}
\item[a.)] This follows from $\int_z^{z+1} \widetilde{b}_r(x) \rd x = 0$ for all $z \in \real$.
\item[b.)] We have $\gamma_2(k,l) = 0$ for $|v-w| > 2$, which follows from Lemma~\ref{lem_eq_gamma2}. Let $r > 2$. Then by repeatedly using (\ref{eq_recursion_gamma}) we can write $\gamma_r(k,l)$ as a sum of $\gamma_2(m_i, n_j)$ for some values $m_i, n_j$, i.e., $\gamma_r(k,l) = \sum_{i,j} a_{i,j} \gamma_2(m_i,n_j)$. But if $|v-w| > 2$, then the difference between the number of digits of $m_i$ and $n_j$ will be bigger than $2$ and hence $\gamma_r(k,l)  = 0$ by Lemma~\ref{lem_eq_gamma2}.
\item[c.)] For $r = 2$ the proof follows again from Lemma~\ref{lem_eq_gamma2}: If $v = 1$ ($w = 1$), then $k'=0$ ($l'=0$ resp.) and we only have the cases $k=l$, $k= l'$ ($l = k'$ resp.), and $k = l''$ ($l = k''$ resp.) for which the result follows. The case $1 \le |v-w| \le 2$ comprises the cases $k' = l$, $k = l'$, $k'' = l$, and $k = l''$. The case $v = w$ can be obtained by considering $k = l$, and $k' = l'$ with $k \neq l$. For $r > 2$, we can again use (\ref{eq_recursion_gamma}) repeatedly to obtain a sum of $\gamma_2(m_i,n_j)$. The result then follows by using Lemma~\ref{lem_eq_gamma2}.
\end{itemize}
\end{proof}

In the following we prove a bound on $|\gamma_r(k.l)|$ for arbitrary $r \ge 2$. We set
\begin{equation*}
\mu_{r, {\rm per}}(k,l) = \max_{0 \le s \le r} \mu_{s, {\rm per}}(k) + \mu_{r-s,{\rm per}}(l).
\end{equation*}


\begin{lemma}\label{lem_bound_gammarkl}
For $r \ge 2$ and $k,l > 0$ we have
\begin{equation*}
|\gamma_{r}(k,l)|  \le  \frac{2 b^{- \mu_{r, {\rm per}}(k,l)}  }{(2\sin\frac{\pi}{b})^{r}} \left(1 + \frac{1}{b} + \frac{1}{b(b+1)} \right)^{r-2}.
\end{equation*}
\end{lemma}

\begin{proof}
For $r = 2$ we use Lemma~\ref{lem_eq_gamma2}, and $|1/2+(\omega_b^{-\kappa}-1)^{-1}|,|\omega_b^{-\kappa}-1|^{-1} \le (2 \sin\frac{\pi}{b})^{-1}$ to obtain the result.

Let now $r > 2$. By taking the absolute value of (\ref{eq_recursion_gamma}) and using the triangular inequality together with   $|1/2+(\omega_b^{-\kappa}-1)^{-1}|,|\omega_b^{-\kappa}-1|^{-1} \le (2 \sin\frac{\pi}{b})^{-1}$ we obtain
\begin{eqnarray}\label{ineq_bound_gammakl}
|\gamma_r(k,l)|  & \le &  \frac{b^{-a_1}}{2\sin \frac{\pi}{b}} \bigg(|\gamma_{r-1}(k',l)| + |\gamma_{r-1}(k,l)| \nonumber \\ && \qquad\qquad  + \sum_{c=1}^\infty \sum_{\vartheta=1}^{b-1} b^{-c} |\gamma_{r-1}(\vartheta b^{a_1+c-1} + k,l)|\bigg).
\end{eqnarray}
By using integration by parts with respect to the variable $y$ in (\ref{def_gammakl}) we obtain a similar formula to (\ref{eq_intparts_bx}). Hence there is also an analogue to (\ref{ineq_bound_gammakl}).

W.l.o.g. assume that $k \ge l$ (otherwise use the analogue to (\ref{ineq_bound_gammakl})) and assume that the result holds for $r-1$.
Then
\begin{eqnarray*}
\lefteqn{ |\gamma_r(k,l)| } \\ & \le & \frac{2 b^{-a_1}}{(2 \sin \frac{\pi}{b})^r} \left(1 + \frac{1}{b} + \frac{1}{b(b+1)}\right)^{r-3} \\ &&   \left(b^{-\mu_{r-1,{\rm per}}(k',l)}  + b^{-\mu_{r-1,{\rm per}}(k,l)}  + (b-1) \sum_{c=1}^\infty b^{-c-\mu_{r-1, {\rm per}}(b^{a_1+c-1} + k, l)}\right).
\end{eqnarray*}
We have $a_1 + \mu_{r-1, {\rm per}}(k',l) = \mu_{r, {\rm per}}(k,l)$, $a_1 + \mu_{r-1, {\rm per}}(k,l) > \mu_{r, {\rm per}}(k,l)$, and $a_1 + \mu_{r-1, {\rm per}}(b^{a_1+c-1}+k,l) = 2a_1 + c + \mu_{r-2,{\rm per}}(k,l) > c + \mu_{r,{\rm per}}(k,l)$. Therefore we obtain
\begin{eqnarray*}
|\gamma_r(k,l)| & \le & \frac{2 b^{-\mu_{r,{\rm per}}(k,l)}}{(2\sin\frac{\pi}{b})^r} \left(1+\frac{1}{b} + \frac{1}{b(b+1)}\right)^{r-3}  \left(1 + \frac{1}{b} + \frac{b-1}{b} \sum_{c=1}^\infty b^{-2c} \right).
\end{eqnarray*}
As $\sum_{c=1}^\infty b^{-2c} = (b^2-1)^{-1}$, the result follows.
\end{proof}

\section{On the Walsh coefficients of functions in Sobolev spaces}\label{sec_RKHS}

In this section we consider functions in reproducing kernel Hilbert spaces. We
consider the Sobolev space $\HH_{r}$ of real valued functions $f:[0,1]\rightarrow \real$, for which $r > 1$, and where the inner product is given by
\begin{equation*}
\langle f, g \rangle_{r} = \sum_{s = 0}^{r-1} \int_0^1 f^{(s)}(x)\rd
x \int_0^1 g^{(s)}(x) \rd x + \int_0^1 f^{(r)}(x) g^{(r)}(x) \rd x,
\end{equation*}
where $f^{(s)}$ denotes the $s$th derivative of $f$ and where
$f^{(0)} = f$. Let $\|f\|_r = \sqrt{\langle f, f \rangle_r}$. The
reproducing kernel (see \cite{aron} for more information about
reproducing kernels) for this space is given by
\begin{eqnarray*}
\KK_{r}(x,y) & = & \sum_{s = 0}^{r} \frac{B_s(x)
B_{s}(y)}{(s!)^2} - (-1)^{r}
\frac{\widetilde{B}_{2r}(x-y)}{(2r)!} \\ & = & \sum_{s=0}^r
b_s(x) b_s(y) - (-1)^r \widetilde{b}_{2r}(x-y),
\end{eqnarray*}
see for example \cite[Section~10.2]{wahba}. It can be checked that
\begin{eqnarray*}
f(y) & = & \langle f, \KK_r(\cdot,y)\rangle_r \\ & = &
\sum_{s=0}^{r} \int_0^1 f^{(s)}(x) \rd x \, b_s(y) -
(-1)^r  \int_0^1 f^{(r)}(x) \widetilde{b}_r(x-y) \rd x. 
\end{eqnarray*}

A bound on the Walsh coefficients of $b_0(y), \ldots, b_r(y)$ can be obtained from Lemma~\ref{lem_bound_walsh_bernoulli}. For the remaining term we use Lemma~\ref{lem_bound_gammarkl}. We have
\begin{equation*}
\widetilde{b}_r(x-y) = \sum_{k,l = 1}^\infty \gamma_r(k,l) \wal_k(x) \overline{\wal_l(y)}
\end{equation*}
and therefore the $m$th Walsh coefficient for the last term is given by
\begin{eqnarray*}
\lefteqn{ (-1)^r \int_0^1 \int_0^1 f^{(r)}(x) \widetilde{b}_r(x-y)\rd x \; \overline{\wal_m(y)} \rd y } \\ & = & (-1)^r \sum_{k,l=1}^\infty \overline{\gamma_r(k,l)} \int_0^1 f^{(r)}(x) \overline{\wal_k(x)} \rd x \int_0^1 \wal_l(y) \overline{\wal_m(y)} \rd y \\ & = & \sum_{k=1}^\infty \gamma_r(m,k) \int_0^1 f^{(r)}(x) \overline{\wal_k(x)} \rd x.
\end{eqnarray*}
We can estimate the absolute value of the last expression by
\begin{eqnarray*}
\lefteqn{\sum_{k=1}^\infty |\gamma_r(m,k)| \int_0^1 |f^{(r)}(x)| \rd x } \\ & \le &  \int_0^1 |f^{(r)}(x)| \rd x  \frac{2}{(2\sin\frac{\pi}{b})^{r}}  \left(1 + \frac{1}{b} + \frac{1}{b(b+1)} \right)^{r-2}   \sum_{\satop{k=1}{\gamma_r(m,k) \neq 0}}^\infty b^{-\mu_{r, {\rm per}}(m,k) }.
\end{eqnarray*}
It remains to prove a bound on the rightmost sum, which we do in the following lemma.

\begin{lemma}
For any $r > 1$ and $m \in \NN$ we have
\begin{equation*}
\sum_{\satop{k=1}{\gamma_r(m,k) \neq 0}}^\infty b^{-\mu_{r,{\rm per}}(m,k)} \le  b^{-\mu_{r, {\rm per}}(m)} \left(3 + \frac{2}{b} + \frac{2b+1}{b-1}\right).
\end{equation*}
\end{lemma}

\begin{proof}
Let $m = \eta_1 b^{e_1-1} + \cdots + \eta_z b^{e_z-1}$, where $0 < \eta_1,\ldots, \eta_z < b$ and $e_1 > \cdots > e_z > 0$. We consider now all natural numbers $k$ for which $\gamma_r(m,k) \neq 0$.
From Lemma~\ref{lem_gamma_0} we know that $\gamma_r(m,k) = 0$ for $|v-z| > r$. Hence we only need to consider the cases where $|v-z| \le r$:
\begin{itemize}
\item $v = \max(z-r,0)$: If $z-r \le 0$, then this case does not occur; otherwise there is only one $k$ for which $\gamma_r(m,k) \neq 0$, and we obtain the summand $b^{-\mu_{r,{\rm per}}(m)}$.

\item $v = \max(z-r+1,0)$: Again if $z-r+1 \le 0$, then this case does not occur; otherwise we can bound this summand from above by $b^{-\mu_{r,{\rm per}}(m)-1}$.

\item $\max(z-r+1,0) < v \le z$: First, let $v = 1$. Then $\kappa_1 = \eta_z$ and $a_1 = e_z$. Therefore $k$ is fixed, $\mu_{r, {\rm per}}(m,k) = \mu_{r,{\rm per}}(m)$, and $b^{-\mu_{r, {\rm per}}(m,k)} = b^{-\mu_{r, {\rm per}}(m)}$.

Let now $v > 1$, which implies $z > 1$ (as $z \ge v$) and $z-v+2 \le r$. In this case $$(a_2, \ldots, a_v, \kappa_2, \ldots, \kappa_v) = (e_{z-v+2}, \ldots, e_z, \eta_{z-v+2}, \ldots, \eta_z).$$ Thus
\begin{eqnarray*}
\mu_{r, {\rm per}}(m,k) & = & \mu_{z-v+1, {\rm per}}(m) + a_1 + \mu_{r-(z-v+2), {\rm per}}(k',k') \\ & \ge & \mu_{r, {\rm per}}(m) + a_1-a_{v-z+r}.
\end{eqnarray*}
Note that $v-z+r > 1$. Let $a'_v = a_1 - a_{v-z+r} > v-z+r-2$. Then the sum over all $k$ for which $1 < v \le z$ is bounded by
\begin{equation*}
b^{-\mu_{r, {\rm per}}(m)} (b-1) \sum_{v = 2}^z  \sum_{a'=v-1}^\infty b^{-a'} \le b^{-\mu_{r, {\rm per}}(m)} \sum_{v=2}^\infty b^{-v+2} \le b^{-\mu_{r, {\rm per}}(m)} \frac{b}{b-1}.
\end{equation*}

\item $z+1 \le v \le z+r-2$: If $z = 1$ then $2 \le v \le r-1$, and, by Lemma~\ref{lem_gamma_0}, we have $\eta_1 = \kappa_v$ and $e_1 = a_v$. In this case $\mu_{r, {\rm per}}(m,k) = \mu_{r, {\rm per}}(k)$ and $\mu_{r, {\rm per}}(k) - \mu_{r, {\rm per}}(m) = (a_1-a_v) + \cdots + (a_v - a_v) + (r-v) (a_v - a_v) = a'_1 + \cdots + a'_{v-1}$, where $a'_i = a_i - a_v$ and $a'_1 > \cdots > a'_{v-1} > 0$. The sum over all $k$ for which $2 \le v \le r-1$, and $\gamma_r(m,k) \neq 0$, is then bounded by
\begin{eqnarray*}
\lefteqn{\sum_{v=2}^{r-1} (b-1)^{v-1} \sum_{a_1 > \cdots > a_{v-1} > a_v = e_1 > 0} b^{-\mu_{r, {\rm per}}(k)} } \\ & \le & b^{-\mu_{r, {\rm per}}(m)} \sum_{v=2}^{r-1} (b-1)^{v-1} \sum_{a'_1 > \cdots > a'_{v-1} > 0} b^{-a'_1 - \cdots - a'_{v-1}} \\ & \le & b^{-\mu_{r, {\rm per}}(m)} \sum_{v=2}^{r-1} b^{-(v-2)} \\ & \le & b^{-\mu_{r, {\rm per}}(m)} \frac{b}{b-1}.
\end{eqnarray*}
For $z > 1$ and $z+1 \le v \le z+r-2$ we have
\begin{equation*}
(a_{v-z+2}, \ldots, a_v, \kappa_{v-z+2}, \ldots, \kappa_v) = (e_2, \ldots, e_z, \eta_2, \ldots, \eta_z)
\end{equation*}
and $v-z+2 \le r$. Thus
\begin{eqnarray*}
\mu_{r,{\rm per}}(m,k)  & = & a_1 + \cdots + a_{v-z+1} + e_1 + \mu_{r-(v-z+2), {\rm per}}(m',m')  \\ & \ge & \mu_{r, {\rm per}}(m) - \mu_{r-1, {\rm per}}(m') + a_1 + \cdots + a_{v-z+1} \\ && + \mu_{r-(v-z+2), {\rm per}}(m',m') \\ & \ge & \mu_{r, {\rm per}}(m) + a'_1 + \cdots + a'_{v-z+1},
\end{eqnarray*}
where $a'_i = a_i - e_{2} = a_i - a_{v-z+2}$ and $a'_1 > \cdots > a'_{v-z+1} > 0$. Thus the sum over all $k$ for which $z+1 \le v \le z+r-2$ and $\gamma_r(m,k) \neq 0$ is bounded by
\begin{eqnarray*}
\lefteqn{ b^{-\mu_{r, {\rm per}}(m)} \sum_{v = z+1}^{z+r-2} (b-1)^{v-z+1} \sum_{a'_1 > \cdots > a'_{v-z+1} > 0} b^{-a'_1 - \cdots - a'_{v-z+1}} }\qquad\qquad\qquad\qquad\qquad\qquad  \\  & \le & b^{-\mu_{r, {\rm per}}(m)} \sum_{v= z+1}^{z+r-2} b^{-1- \cdots - (v-z)} \\ & \le & \frac{b^{-\mu_{r, {\rm per}}(m)}}{b-1}.
\end{eqnarray*}

\item $v = z+r$: In this case $\mu_{r,{\rm per}}(m,k) = a_1 + \cdots + a_r - \mu_{r, {\rm per}}(m) + \mu_{r, {\rm per}}(m)$, where $\mu_{r, {\rm per}}(m) \le r a_{r+1}$. Thus $a_1 + \cdots + a_r - \mu_{r, {\rm per}}(m) \ge (a_1-a_{r+1}) + \cdots + (a_r- a_{r+1})$ and $a_1 > \cdots > a_r > a_{r+1}$. Hence, the sum over all $k$ for which $v = z+r$ is bounded by
\begin{equation*}
(b-1)^r b^{-\mu_{r,{\rm per}}(m)} \sum_{a_1 > \cdots > a_r > 0} b^{-a_1 - \cdots - a_r} \le  b^{-\mu_{r, {\rm per}}(m)} b^{-r (r-1)/2}.
\end{equation*}

\item $v = z+r-1$: In this case $\mu_{r, {\rm per}}(m,k) = a_1 + \cdots + a_r - \mu_{r, {\rm per}}(m) + \mu_{r, {\rm per}}(m)$, where now $a_r = e_1$ and $\kappa_r = \eta_1$ are fixed. Hence, the sum over all $k$ for which $v = z+r-1$ is bounded by
\begin{equation*}
(b-1)^{r-1} b^{-\mu_{r,{\rm per}}(m)} \sum_{a_1 > \cdots > a_{r-1} > 0} b^{-a_1 - \cdots - a_{r-1}}  \le  b^{-\mu_{r, {\rm per}}(m)} b^{-(r-1)(r-2)/2}.
\end{equation*}
\end{itemize}

By summing up the bounds obtained for each case, we obtain the result.
\end{proof}

This implies the following theorem.

\begin{theorem}\label{thm_sobspace}
Let $r > 1$. Then for any $k \in \NN$ we have
\begin{eqnarray*}
|\hat{f}(k)| & \le &  \sum_{s=v}^r \left|\int_0^1 f^{(s)}(x) \rd x \right| \; \frac{b^{-\mu_{s,{\rm per}}(k)}}{(2\sin\frac{\pi}{b})^s} \; \left(1 + \frac{1}{b} + \frac{1}{b(b+1)}\right)^{\max(0,s-2)} \\ && + \int_0^1 |f^{(r)}(x)| \rd x \; \frac{2  b^{-\mu_{r, {\rm per}}(k)}}{(2\sin\frac{\pi}{b})^r} \left(1 + \frac{1}{b} + \frac{1}{b(b+1)}\right)^{r-2} \\ && \qquad\qquad\qquad\qquad\qquad\qquad \times \left(3 + \frac{2}{b} + \frac{2b+1}{b-1}\right),
\end{eqnarray*}
for all $f \in \HH_r$, where for $v > r$  the empty sum $\sum_{s=v}^r$ is defined as $0$.
\end{theorem}

\begin{remark}
This theorem can easily be generalized to tensor product spaces, for
which the reproducing kernel is just the product of the
one-dimensional kernels.
\end{remark}

\section{On the Walsh coefficients of smooth periodic functions}\label{sec_periodicRKHS}

We consider a subset of the previous reproducing kernel Hilbert space, namely, let $\HH_{r, {\rm per}}$ be the space of all functions $f \in \HH_r$ which satisfy the condition $\int_0^1 f^{(s)}(x) \rd x = 0$ for $0 \le s < r$. This space also has a reproducing kernel, which is given by
\begin{equation*}
\KK_{r, {\rm per}}(x,y)  =  (-1)^{r+1}
\frac{\widetilde{B}_{2r}(x-y)}{(2r)!}  = (-1)^{r+1} \; \widetilde{b}_{2r}(x-y).
\end{equation*}
The inner product is given by
\begin{equation*}
\langle f, g \rangle_{r, {\rm per}} = \int_0^1 f^{(r)}(x) g^{(r)}(x) \rd x.
\end{equation*}
We also have the representation
\begin{equation*}
f(y) =  (-1)^{r+1} \int_0^1 f^{(r)}(x) \widetilde{b}_r(x-y) \rd x.
\end{equation*}

For this space we can obtain an analogue to Theorem~\ref{thm_sobspace}.
\begin{theorem}\label{thm_sobspace_periodic}
Let $r > 1$. Then for any $k \in \NN$ we have
\begin{eqnarray*}
|\hat{f}(k)| & \le & \int_0^1 |f^{(r)}(x)| \rd x  \\ && \times \frac{2  b^{-\mu_{r, {\rm per}}(k)}}{(2\sin\frac{\pi}{b})^r} \left(1 + \frac{1}{b} + \frac{1}{b(b+1)}\right)^{r-2} \left(\frac{7}{2} + \frac{2}{b} + \frac{2b+1}{b-1}\right),
\end{eqnarray*}
for all $f \in \HH_{r,{\rm per}}$.
\end{theorem}

\begin{remark}
This theorem can easily be generalized to tensor product spaces, for
which the reproducing kernel is just the product of the
one-dimensional kernel.
\end{remark}

\begin{remark}
For $2 \le b \le 4$ we have $(2\sin\frac{\pi}{b})^{-1} (1 + \frac{1}{b} + \frac{1}{b(b+1)}) < 1$, and so, for these cases, the constants in the theorems above decrease as $v,r$ increase.
\end{remark}

\begin{remark}
Because the Walsh coefficients considered in this paper converge
fast, the Walsh series for functions $f$ with smoothness $r > 1$
converges absolutely (we have $\sum_{k=0}^\infty b^{-\mu_{r, {\rm
per}}(k)} \le \sum_{k=0}^\infty b^{-\mu_r(k)} < \infty$ for $r > 1$
and $\sum_{k=0}^\infty b^{-a_1-\lambda a_2} < \infty$ for $\lambda >
0$) and we have (see \cite{Dick06})
$$f(x) = \sum_{k=0}^\infty \hat{f}(k) \wal_k(x) \quad\mbox{for } 0
\le x < 1.$$
\end{remark}

\section{Lower bounds}

Fine~\cite[Theorem~VIII]{Fine} proved that the only absolutely
continuous functions whose Walsh coefficients decay faster than
$1/k$ are the constants.

This result can be extended in the following way. Let $f$ have
smoothness $r > 1$ (i.e., the $\lfloor r \rfloor$th derivative has
at least bounded variation of order $r-\lfloor r \rfloor$). Let
$k^{(1)} = k'$, $k^{(2)} = k''$ and $k^{(s)} = k^{(s-1)} - \kappa_s
b^{a_s-1} = \kappa_{s+1} b^{a_{s+1}-1} + \cdots + \kappa_v
b^{a_v-1}$ for $1 \le s < v$ and let $k^{(v)} = 0$.

Using $\int_0^1 f(x) \overline{\wal_k(x)}\rd x = -\int_0^1 f'(x)
J_k(x) \rd x$ and Lemma~\ref{lem_Jk} we have
\begin{eqnarray*}
\hat{f}(k) & = & -b^{-a_1} (1-\omega_b^{-\kappa_1})^{-1}
\hat{f'}(k') - b^{-a_1} (1/2+(\omega_b^{-\kappa_1})^{-1})
\hat{f'}(k)
\\ && - b^{-a_1} \sum_{c=1}^\infty \sum_{\vartheta=1}^{b-1} b^{-c}
(\omega_b^{\vartheta}-1)^{-1} \hat{f'}(\vartheta b^{a_1+c-1}+k),
\end{eqnarray*}
where $\hat{f'}(k)$ denotes the $k$th Walsh coefficient of $f'$.

Using Lemma~\ref{lem_boundVf}  we obtain for fixed $k^{(1)}$ and
$\kappa_1$ that
\begin{equation*}
\lim_{a_1 \rightarrow \infty} b^{a_1} \hat{f}(k) =
-(1-\omega_b^{-\kappa_1})^{-1} \hat{f'}(k^{(1)}).
\end{equation*}
Applying this result inductively we obtain for $s \le \min(\lfloor r
\rfloor,v)$ and fixed $k^{(s)}$ and $\kappa_1,\ldots, \kappa_s$ that
\begin{eqnarray*}
\lefteqn{ \lim_{a_s\rightarrow \infty} b^{a_s} \cdots \lim_{a_2
\rightarrow \infty} b^{a_2} \lim_{a_1\rightarrow \infty} b^{a_1}
\hat{f}(k) } \\ & = & - \lim_{a_s\rightarrow \infty} b^{a_s} \cdots
\lim_{a_2 \rightarrow \infty} b^{a_2} (1-\omega_b^{-\kappa_1})^{-1}
\hat{f'}(k^{(1)}) \\ & = & (-1)^s \prod_{s'=1}^s
(1-\omega_b^{-\kappa_{s'}})^{-1} \hat{f^{(s)}}(k^{(s)}).
\end{eqnarray*}
This implies that if $\hat{f}(k)$ decays faster than $b^{-a_1-\cdots
- a_s}$ for all $k = \kappa_1 b^{a_1-1} + \cdots + \kappa_v
b^{a_v-1}$ with $0 < \kappa_1,\ldots, \kappa_v < b$ and $v \ge s$,
then $\hat{f^{(s)}}(k) = 0$ for all $k$ and therefore $f^{(s)} = 0$,
i.e., $f$ is a polynomial of degree at most $s-1$.

\begin{theorem}
Let $f$ have smoothness $r > 1$. Then if for some $1 \le s \le r$,
$\hat{f}(k)$ decays faster than $b^{-a_1-\cdots - a_s}$ for all $k =
\kappa_1 b^{a_1-1} + \cdots + \kappa_v b^{a_v-1}$ with $0 <
\kappa_1,\ldots, \kappa_v < b$ and $v \ge s$, i.e.,
\begin{equation*}
 \lim_{a_s\rightarrow \infty} b^{a_s} \cdots \lim_{a_2
\rightarrow \infty} b^{a_2} \lim_{a_1\rightarrow \infty} b^{a_1}
\hat{f}(k)  = 0 \quad \mbox{for all } k \mbox{ with } v \ge s,
\end{equation*}
then $f$ is a polynomial of degree at most $s-1$.
\end{theorem}

\end{document}